\documentclass[12pt, twoside, reqno]{amsart}

\usepackage{amsmath}
\usepackage{amsfonts}
\usepackage{amssymb}
\usepackage{amsthm}
\usepackage{hyperref}
\usepackage{graphicx}
\usepackage{cite}
\usepackage{caption}

\usepackage[top=2.5cm, bottom=2.5cm, left=2.5cm, right=2.5cm, headsep=0.2in]{geometry}

\usepackage{fancyhdr}

\DeclareMathOperator{\adj}{adj}
\DeclareMathOperator{\im}{Im}
\DeclareMathOperator{\re}{Re}

\theoremstyle{plain}
\newtheorem{theorem}{Theorem}[section]

\newtheorem{lemma}[theorem]{Lemma}
\newtheorem{rem}[theorem]{Remark}
\newtheorem{prop}[theorem]{Proposition}

\newtheorem{que}[theorem]{Question}

\numberwithin{equation}{section}

\newcommand{\Lapl}{\mathcal{L}}
\renewcommand{\div}{\text{div}}

\begin{document}

\title{Harmonic determinants and unique continuation}
\author[M. Ceki\'{c}]{Mihajlo Ceki\'{c}}
\address{Max-Planck Institute for Mathematics, Vivatsgasse 7, 53111, Bonn, Germany}
\email{m.cekic@mpim-bonn.mpg.de}

\begin{abstract}
We give partial answers to the following question: if $F$ is an $m$ by $m$ matrix on $\mathbb{R}^n$ satisfying a second order linear elliptic equation, does $\det F$ satisfy the strong unique continuation property? We give counterexamples in the case when the operator is a general non-diagonal operator and also for some diagonal operators. Positive results are obtained when $n = 1$ and any $m$, when $n = 2$ for the Laplace-Beltrami operator and also twisted with a Yang-Mills connection. Reductions to special cases when $n = 2$ are obtained. The last section considers an application to the Calder\'on problem in 2D based on recent techniques.
\end{abstract}

\maketitle

\section{Introduction}\label{sec1}
The strong unique continuation property (SUCP) for second order elliptic equations with smooth coefficients is well-known. It asserts that a solution vanishing to infinite order at a point must entirely vanish; on the other hand the weak unique continuation principle (WUCP) asserts that a function vanishing on an open subset, must vanish entirely. Clearly SUCP implies WUCP. There are a few known approaches: by Carleman estimates (see \cite{KT07} for a survey) and the frequency method (see \cite{GL} for this approach). It is not difficult to see from this that elliptic systems with diagonal principal part also satisfy the SUCP (see e.g. \cite{cek1}). One reason to be interested in this property is that the zero sets of such solutions have a suitable structure: they are countably $(n-1)$-rectifiable \cite{bar}, i.e. covered by a countable union of codimension one smooth submanifolds.

Consider a domain $\Omega \subset \mathbb{R}^n$ equipped with a positive definite (uniformly) $n$ by $n$ matrix function $a^{ij}$. Suppose $F: \Omega \to \mathbb{C}^{m \times m}$ is a solution to 
\begin{equation}\label{eqn}
PF(x) = -\partial_i\big(a^{ij}\partial_jF\big)(x) + L(x, F(x), dF(x)) = 0
\end{equation}
for $x \in \Omega$, where $L$ is a smooth matrix function, linear in $F$ and $dF$ entries. We will sometimes write $(\Omega, g) \subset \mathbb{R}^n$ when $a^{ij}$ comes from a Riemannian metric $g$ on $\Omega$, so that $a^{ij} = \frac{g^{ij}}{|g|}$ represents the Laplace-Beltrami operator, where $|g| = \sqrt{\det g}$. We address the question:

\begin{que}\label{mainque}
Does the SUCP hold for $\det F$, where $F$ satisfies \eqref{eqn}? If not, does the WUCP hold?
\end{que} 
Here are a few starting remarks -- firstly, in \cite{cek1} we notice that if $g$ is analytic and so are the coefficients of $L$, then by the classical theory so are the entries of $F$ and consequently, so is $\det F$ and the SUCP holds. Secondly, the obvious approach to produce an elliptic equation that $\det F$ satisfies does not seem to work (if we compute $\Delta_g \det F$ we obtain a function of $F$ and $dF$).

Some further motivation is also due. Except that this problem is a natural one to consider when studying systems, the author is motivated by the case of the connection Laplacian $P = d_A^*d_A$, where $d_A = d + A$ is a covariant derivative, $A$ is the $m \times m$ connection matrix of $1$-forms and $d_A^*$ is the formal adjoint of $d_A$ in the natural inner products. For, this problem appeared to be one of the crucial ones when studying the inverse problem of Calder\'on for Yang-Mills connections \cite{cek1} -- there, the gauge relating two connections $A$ and $B$ which have the same local Dirichlet-to-Neumann map was shown to be $H = FG^{-1}$ where $d_A^*d_AF = d_B^*d_B G = 0$ if $m = 1$ or for any $m$ if the metric is analytic. The tactic is to use unique continuation near the boundary and to analyse the zero set of $\det G$ to further extend $H$ smoothly inside the manifold. So the unique continuation property for $\det G$ for $m > 1$ comes into focus.

We propose a few approaches to this problem. In 2D we may use a set of special coordinates which reduce us to the case of special matrix $a^{ij}$; then by quotienting out one entry we are further reduced to the case where one of the entries is equal to $1$. Another simple technique is to compare the leading order Taylor coefficients of the entries, which we employ in the case $m = 2$. We also give several negative results for non-diagonal systems and some for diagonal systems; most of them are based on the simple observation that a PDE can be viewed as an equation for its coefficients.

Unless otherwise stated, the coefficients of the equations and the solutions are assumed to be in $C^\infty$. In the following Theorem, we summarise the positive and negative results for the SUCP for $\det F$ that are proven in this paper. As far as I know, this is the first time someone considered this problem and so the results are new in this sense.

\begin{theorem}\label{mainthm'}
The following table summarises the answers to Question \ref{mainque} that were proved in this paper for varying operator $P$, $m$ and $n$ and includes some open cases:
\begin{center}
    \begin{tabular}{  l | l | l | p{6cm} }
    Operator $P =$ & $m$ & $n$ & SUCP: Yes, No or Unknown? \\ \hline
    $\frac{d^2}{dt^2}\times Id + \begin{pmatrix}
0 & 0\\
-\frac{d^2c}{dt^2} \frac{d}{dt} & 0
\end{pmatrix}$ & $m \geq 2$ & $n = 1$ & No (counterexample to WUCP).\\ \hline
    $\Delta_g \times Id + \begin{pmatrix}
X_{11}^y \partial_y & X_{12}\\
\partial_x & X_{22}
\end{pmatrix}$ & $m \geq 2$ & $n \geq 1$ & No (counterexample to WUCP).\\ \hline
    $a\partial_1^2 + b\partial_2^2 + c\partial_1 + d\partial_2$ & $m \geq 2$ & $n \geq 2$ & No (counterexample to WUCP). \\ \hline
    $-\partial_i(a^{ij}\partial_j + b^i)$ & $m \geq 2$ & $n \geq 2$ & No (counterexample to WUCP). \\ \hline
    Analytic coefficients & $m \geq 1$ & $n \geq 1$ & Yes. \\ \hline
    $\Delta_g \times Id$ & $m \geq 1$ & $n = 2$ & Yes. \\ \hline
    $d_A^*d_A$ (for $A$ Yang-Mills) & $m \geq 1$ & $n = 2$ & Yes. \\ \hline
    $\frac{d^2}{dt^2} + a\frac{d}{dt} + b$ & $m \geq 1$ & $n = 1$ & Yes. \\ \hline
    $\Delta_g \times Id$ & $m \geq 2$ & $n \geq 3$ & Unknown. \\ \hline
    $\partial_i(a^{ij}\partial_j) \times Id$ & $m = 2$ & $n = 2$ & Unknown if $\det A \neq 1$ or $A \neq A^T$. \\ \hline
    \end{tabular}
\end{center}
\end{theorem}
\vspace{3mm}

We expect the last two SUCP properties in the table above to be false, but it seems difficult to construct direct counterexamples and we do not have a proof of this fact.

Next, we use the SUCP result for the operator $P = d_A^*d_A$ in the following application to the Calder\'on problem for connections, by using the techniques from \cite{cek1}. As explained above, the zero set of $\det G$ is then countably $(n - 1)$-rectifiable and we may re-run the proof of Theorem 1.2. in \cite{cek1}. There are slight complications near the zero set, since the order of degeneracy of $\det G$ can be high, but we work around this by going to a harmonic coordinate system near such a point.

Before stating the theorem, let us briefly recall the definition of a Yang-Mills connection, which are connections important in physics and geometry -- see \cite{DK} for more details. A unitary connection $A$ on a Hermitian vector bundle $E$ over a Riemannian manifold $(M, g)$ is called \emph{Yang-Mills} if it is the critical point of the Yang-Mills functional $F_{YM}$:
\begin{align}\label{YMeqn}
F_{YM}(A) = \int_M |F_A|^2 dvol_g
\end{align}
where $F_A = dA + A\wedge A$ is (locally) the curvature two form. Alternatively, it satisfies the equation
\begin{align}\label{YMeqn}
D_A^*F_A = 0
\end{align}
where $D_A S = dS + [A, S]$ is the induced covariant derivative on the endomorphism bundle End$E$.

\begin{theorem}\label{mainthm}
Let $(M, g)$ be a compact smooth $2$-dimensional Riemannian manifold with non-empty boundary and let $A$ and $B$ be two Yang-Mills connections over $M \times \mathbb{C}^m$ (for $m \in \mathbb{N}$). Further, let $\Gamma \subset \partial M$ be a non-empty open subset of the boundary. Then $\big(\Lambda_A f\big)|_{\Gamma} = \big(\Lambda_B f\big)|_{\Gamma}$\footnote{In this setting, recall that the \emph{Dirichlet-to-Neumann map} $\Lambda_A$ is defined by applying the covariant normal derivative at the boundary to the solution of the corresponding Dirichlet problem $\Lambda_A(f) = d_A(u)(\nu)|_{\partial M}$, where $\nu$ is the outer normal to $\partial M$ and $u$ solves the Dirichlet problem $d_A^*d_A u = 0$ with $u|_{\partial M} = f$.} for all $f \in C_0^\infty(\Gamma, \mathbb{C}^m)$ implies the existence of a unitary matrix function $H \in C^\infty(M, \mathbb{C}^{m \times m})$ with $H|_{\Gamma} = Id$ and $H^*A = B$.
\end{theorem}

Note that for $\Gamma = \partial M$, i.e. full data, the above Theorem follows from the work in \cite{AGTU}, which recovers a general matrix potential and the connection on an arbitrary vector bundle up to gauges with a different technique based on the Complex Geometric Optics (CGO) solutions. One advantage of Theorem \ref{mainthm} is that it holds for partial data. Also, it extends the new technique of \cite{cek1} based on analysing the zero set, which gives hope this technique can be extended to more general contexts.

Finally, we note there is a different, but related variation of Quesion \ref{mainque} where one considers the Jacobian of a system and its zero set. As observed in \cite{B}, this is of some importance in hybrid inverse problems. For example, in \cite{AN01}, in 2D, the authors consider the Jacobian $J = \det DU$ formed by solutions to $\div A\nabla u_i = 0$ for $i = 1, 2$ (these are also called $A$-harmonic functions -- see Section \ref{sec5}), where $U = (u_1, u_2)^T$. They state conditions on the boundary values of $U$ under which an $A$-harmonic extension of $U$ to the domain is univalent (injective) and provide local bounds on $\log J$.\footnote{Interestingly, they derive an elliptic equation for $J$ in this case. This seems unavailable for our problem.} See references in \cite{AN01, B} for more about this problem and its applications (also in higher dimensions).

The paper is organised as follows. In Section \ref{sec2} we consider counterexamples in the non-diagonal case and also to the general diagonal case, as stated in Theorem \ref{mainthm'}. In Section \ref{sec3} we consider positive results in 1D. In Section \ref{sec4} we consider the $n = 2$ case in more detail. More precisely, we prove a few positive results, including the case of $P = \Delta_g$ and arbitrary $m$, see Theorem \ref{isothermal2dsucp}; we also prove a slightly more general result for $m = 2$. Furthermore, we reduce the problem to a simpler form for $m = 2$ by using properties of harmonic polynomials in 2D and a reduction lemma: see Proposition \ref{mainreduction}. Some further reductions in 2D are given in Section \ref{sec5}, based on the theory of quasiconformal maps. In Section \ref{sec6} we prove a positive result in two dimensions for the connection Laplacian operator twisted with a Yang-Mills connection. Finally, in Section \ref{sec7} we consider an application to the Calder\'on problem in two dimensions, based on the recent techniques in \cite{cek1}. In Appendix \ref{secapp} we prove a simple geometric lemma and a result on products of harmonic polynomials in two dimensions that we need.

\vspace{4mm}
\textbf{Acknowledgements.} I would like to thank Herbert Koch for helpful discussions and to Gabriel Paternain for useful comments. Also, the author thanks the Max-Planck Institute for Mathematics for financial support.

\section{Negative results}\label{sec2}

We start with the negative results and by showing what we cannot expect to hold.

\begin{theorem}[Counterexample]\label{cex}
Assume $g = g_{eucl}$ is the Euclidean metric and $0 \in \Omega \subset \mathbb{R}^2$. Let $c: \Omega \to \mathbb{R}$ be a smooth function to be specified later. Define
\begin{equation}
X = \begin{pmatrix}
X_{11}^y \partial_y & 0\\
\partial_x & 0
\end{pmatrix}
\end{equation}
be a first order matrix derivative, where 
\begin{equation}\label{X11'}
X_{11}^y(x, y) = \frac{\partial_x \Delta_g c + \int_0^x \partial_y^2 \Delta_g c(t, y) dt}{1 - \int_0^x \partial_y \Delta_g c(t, y) dt}
\end{equation}
Moreover, define
\begin{equation}
b(x, y) = y - \int_0^x \Delta_g c(t, y) dt
\end{equation}
and let
\begin{equation}\label{Fdef}
F := \begin{pmatrix}
1 & b\\
0 & c
\end{pmatrix}
\end{equation}
Then $F$ satisfies (here $X$ acts by matrix multiplication)
\begin{equation}\label{Feqn}
\Delta_g F + XF = 0
\end{equation}
and also $\det F = c$. So, by allowing $c = e^{-\frac{1}{|x|^2}}$ (vanishes to infinite order at zero) or letting $c$ to be a bump function equal zero in a neighbourhood of zero, we obtain respectively a counterexample to the SUCP and the WUCP.
\end{theorem}
\begin{proof}
We have $\Delta_g = -(\partial_x^2 + \partial_y^2)$ and we are left to verify a simple computation. Note that $X_{11}^y$ in \eqref{X11'} is well defined in a neighbourhood of zero and near the zero set of $c$ in the second case. We can easily check that, from the definitions
\begin{align*}
\Delta_g b\; +& &X_{11}^y \partial_y b& &= 0 &\iff -\partial_x \Delta_g c - \int_0^x \partial_y^2 \Delta_g c + \frac{\partial_x \Delta_g c + \int_0^x \partial_y^2 \Delta_g c}{1 - \int_0^x \partial_y \Delta_g c} \cdot (1 - \int_0^x \partial_y  \Delta_g c) = 0\\
\Delta_g c\;+& &\partial_x b& &= 0 &\iff \Delta_gc - \Delta_g c = 0
\end{align*}
\end{proof}

This is one of the simplest counterexamples we could find. We can upgrade it to:

\begin{theorem}\label{cex2}
In the same setting as Theorem \ref{cex}, we let
\begin{equation}
X = \begin{pmatrix}
X_{11}^y \partial_y & X_{12}\\
\partial_x & X_{22}
\end{pmatrix}
\end{equation}
where $X_{12}$ and $X_{22}$ are smooth first order derivatives. Then by letting
\begin{equation}\label{X11}
X_{11}^y(x, y) = \frac{\partial_x (\Delta_g c + X_{22}c) - X_{12}c + \int_0^x \partial_y^2 \Delta_g c(t, y) dt}{1 - \int_0^x \partial_y \Delta_g c(t, y) dt}
\end{equation}
and
\begin{equation}
b(x, y) = y - \int_0^x \big(\Delta_g c(t, y) + X_{22}c(t, y)\big) dt
\end{equation}
we obtain the solution $F$ from \eqref{Fdef} satisfying equation \eqref{Feqn} and so we generalise the counterexample to this case.
\end{theorem}

\begin{rem}\rm
Note that Theorem \ref{cex2} provides us in particular with a counterexample to SUCP and WUCP for $X$ symmetric (Hermitian) or anti-symmetric (skew-Hermitian). This is relevant for the twisted Laplacian operator which is of the form 
\begin{equation}
d_A^*d_A = \Delta_g - 2g^{ij} A_i \partial_j + d^*A - g^{ij}A_iA_j
\end{equation}
Here $A = A_i dx^i$ is the connection one form, $g^{ij}$ is the inverse of the metric matrix $g_{ij}$ and $d^*$ is the co-differential. If the connection $A$ is unitary, then $A_i$ is skew-Hermitian. What the previous theorem is telling us is that we should not expect the SUCP to hold for $d_A^*d_A$ for $n \geq 2$ and general $A$. The fact that for $A$ Yang-Mills and $n = 2$ (c.f. Section \ref{sec6}) we have SUCP is due to the special analytical properties in suitable gauges in 2D, so we do not expect the SUCP to hold even for Yang-Mills connections and $n \geq 3$, but this remains open.
\end{rem}

In the similar vein as the counterexamples above, we give a simple counterexample in the $1$-dimensional case. More precisely, we have:

\begin{prop}\label{cexn=1} Let us define the smooth matrix function, for a smooth $c: \mathbb{R} \to \mathbb{R}$
\[F(t) = \begin{pmatrix}
1 & t\\
0 & c(t)
\end{pmatrix}\]
Furthermore, let us define the first order smooth matrix derivative
\[X (t)= \begin{pmatrix}
0 & 0\\
-\frac{d^2c}{dt^2} \frac{d}{dt} & 0
\end{pmatrix}\]
Then $F$ satisfies $\frac{d^2 F}{dt^2} + XF = 0$ and we have $\det F = c$. By letting $c$ to be an infinitely vanishing function at zero and a bump function equal vanishing near zero, we obtain counterexamples to the SUCP and WUCP, respectively.
\end{prop}
\begin{proof}
Immediate from the construction.
\end{proof}

The next counterexample rules out even \emph{diagonal} operators in dimension $4$. It is based on the simple idea that a solution to a PDE can be viewed as an equation in the coefficients and some linear algebra. The more coefficients we have, the more space we have to prescribe the solutions and then determine the coefficients -- this is why dimension $4$ is useful.

\begin{theorem}[Counterexample in the diagonal case in 4D]\label{cexn=4}
There exist an $\varepsilon > 0$ and smooth, positive and real coefficient functions $a, b, c, d$ on $B_\varepsilon$ and smooth functions $f_1, f_2, f_3$ on $B_{2\varepsilon}$, such that for $\Lapl: = (a \partial_1^2 + b\partial_2^2 + c\partial_3^2 + d\partial_4^2)$, we have
\begin{align}\label{eqncexn=4}
\Lapl f_i = (a \partial_1^2 + b\partial_2^2 + c\partial_3^2 + d\partial_4^2)f_i = 0
\end{align}
for $i = 1, 2, 3$. Also, we have $f_1 f_2 = f_3$ on $B_{\varepsilon}$, but $f_1 f_2 \neq f_3$ on $B_{2\varepsilon} \setminus B_{\varepsilon}$.

Therefore, $F := \begin{pmatrix}
f_3 & f_2\\
f_1 & 1
\end{pmatrix}$ satisfies $\Lapl F = 0$ and $\det F = 0$ on $B_{\varepsilon}$, but $\det F \neq 0$ on $B_{2\varepsilon} \setminus B_{\varepsilon}$; so the WUCP fails in this case.
\end{theorem}
\begin{proof}
Note that the equation \eqref{eqncexn=4} holds for $i = 1, 2, 3$ if and only if the following matrix equation holds:
\begin{align}\label{matrixeqn}
\begin{pmatrix}
\partial_1^2 f_1 & \partial_2^2 f_1 & \partial_3^2 f_1 & \partial_4^2 f_1\\
\partial_1^2 f_2 & \partial_2^2 f_2 & \partial_3^2 f_2 & \partial_4^2 f_2\\
\partial_1^2 f_3 & \partial_2^2 f_3 & \partial_3^2 f_3 & \partial_4^2 f_3
\end{pmatrix} \begin{pmatrix}
a\\
b\\
c\\
d
\end{pmatrix} = 0
\end{align}
at all points $p$ in the domain of definition. Note that this $3 \times 4$ matrix has nullity $\geq 1$ and so there is always a non-zero solution at each point $p$. Let us choose auxiliary functions 
\begin{align*}
g_1 = x^2 - y^2 + t + x, \quad g_2 = x^2 - z^2 + t - x \quad \text{ and } \quad g_3 = g_1g_2
\end{align*}
With this chose, we have full rank at $p = 0$ ($f_1 = g_1$, $f_2 = g_2$, $f_3 = g_3$) and so we have the non-zero kernel spanned with $a = b = c = d = 1$. Since the rank of the $3 \times 4$ matrix from \eqref{matrixeqn} must be full in a neighborhood of zero (determinant of a $3\times 3$ minor is non-zero), there exists and $\epsilon > 0$ such that on $B_\varepsilon$ we have a smooth choice of solutions to \eqref{matrixeqn} with $a(0) = b(0) = c(0) = d(0) = 1$ and for some small $\delta > 0$
\[\min_{B_\varepsilon}\{a, b, c, d\} \geq 1 - \delta\] 
Now choose smooth extensions $f_1, f_2, f_3$ to be such that they agree with $g_1, g_2, g_3$ on $B_\varepsilon$ and such that $f_1 f_2 \neq f_3$ on $B_{2\varepsilon} \setminus B_\varepsilon$ (e.g. multiply with a bump function)
\[\min_{B_{2\varepsilon}}\{a, b, c, d\} \geq \frac{1}{2}\]
This can be done for $\varepsilon$ and $\delta$ small enough and a good choice of extensions. This finishes our construction.
\end{proof}

Note that the above construction also gives a counterexample in any dimension $n \geq 4$ and the size of the matrix $m \geq 2$. The next Proposition tells us we can do slightly better by introducing off-diagonal terms in dimension $3$:

\begin{prop}[Counterexample in the diagonal case in 3D]
There exist an $\varepsilon > 0$ and smooth, positive and real coefficient functions $a, b, c, d$ on $B_{2\varepsilon}$ and smooth functions $f_1, f_2, f_3$ on $B_{2\varepsilon}$, such that the operator $\Lapl: = (a \partial_1^2 + b\partial_2^2 + c\partial_3^2 + 2d\partial_1\partial_2)$ is (strongly) elliptic and we have
\begin{align}\label{eqncexn=4}
\Lapl f_i = (a \partial_1^2 + b\partial_2^2 + c\partial_3^2 + 2d\partial_1 \partial_2)f_i = 0
\end{align}
for $i = 1, 2, 3$. Moreover, we have $f_1 f_2 = f_3$ on $B_{\varepsilon}$, but $f_1 f_2 \neq f_3$ on $B_{2\varepsilon} \setminus B_{\varepsilon}$.

Therefore, $F := \begin{pmatrix}
f_3 & f_2\\
f_1 & 1
\end{pmatrix}$ satisfies $\Lapl F = 0$ and $\det F = 0$ on $B_{\varepsilon}$, but $\det F \neq 0$ on $B_{2\varepsilon} \setminus B_{\varepsilon}$; so the WUCP fails in this case.
\end{prop}
\begin{proof}
Similar to the proof of the previous theorem. We choose the following functions:
\[g_1 = x^2 - y^2 + x, \quad g_2 = x^2 - z^2 + x - 2y \quad \text{ and } \quad g_3 = g_1 g_2\]
Note that this yields $a(0) = b(0) = c(0) = 1$ and $d(0) = \frac{1}{2}$ to be the solution as in \eqref{matrixeqn}. From this point, the argument works the same. 
\end{proof}

Finally, we show that if we introduce some linear terms, we can go to two dimensions, as well.

\begin{theorem}[Counterexample in the diagonal case in 2D]\label{cexn=4}
There exist an $\varepsilon > 0$ and smooth, positive and real coefficient functions $a, b, c, d$ on $B_\varepsilon \subset \mathbb{R}^2$ and smooth functions $f_1, f_2, f_3$ on $B_{2\varepsilon}$, such that for $\Lapl: = (a \partial_1^2 + b\partial_2^2 + c\partial_1 + d\partial_2)$, we have
\begin{align}\label{eqncexn=4}
\Lapl f_i = (a \partial_1^2 + b\partial_2^2 + c\partial_1 + d\partial_2)f_i = 0
\end{align}
for $i = 1, 2, 3$. Also, we have $f_1 f_2 = f_3$ on $B_{\varepsilon}$, but $f_1 f_2 \neq f_3$ on $B_{2\varepsilon} \setminus B_{\varepsilon}$.

Therefore, $F := \begin{pmatrix}
f_3 & f_2\\
f_1 & 1
\end{pmatrix}$ satisfies $\Lapl F = 0$ and $\det F = 0$ on $B_{\varepsilon}$, but $\det F \neq 0$ on $B_{2\varepsilon} \setminus B_{\varepsilon}$; so the WUCP fails in this case.
\end{theorem}
\begin{proof}
The tactics is the same as before, but we now let
\[g_1 = x^2 + x + y, \quad g_2 = x^2 + x - y \quad \text{ and } \quad g_3 = g_1 g_2\]
Then at the origin we have the solution to \eqref{matrixeqn} given by $a(0) = b(0) = 1$, $c(0) = -2$ and $d(0) = 0$. Now we extend these functions to $f_1, f_2$ and $f_3$ and note that the ellipticity is preserved by small perturbations.
\end{proof}

Finally, we give a counterexample for the WUCP in the case of a divergence type operator (with a zero order term under the divergence sign) and a $2$ by $2$ matrix in 2D. The approach combines the ideas above in Theorem \ref{cexn=4} and the reduction techniques of Alessandrini \cite{ales} and Schulz \cite{Schulz}. See also Section \ref{reductions} below on these reduction techniques. The idea is to generate solutions to $Lu = -\div(A\nabla u + b \cdot u) + C\cdot \nabla u + du$ using the techniques above and then use the reduction techniques to get rid of the $C$ and $d$ coefficients.

We start by stating an algebraic Lemma (c.f. Lemma \ref{lem1}):

\begin{lemma}\label{reductionlemma}
Let $A$ be a symmetric matrix, $C$ and $b$ vector functions and $d$ a scalar function (all smooth) on $\mathbb{R}^n$. Consider the operator
\begin{align*}
\Lapl u = -\partial_i \big(a^{ij} \partial_j u + b^i u\big) + c^i \partial_i u + du
\end{align*}
Assume $\Lapl \varphi = 0$ and $\Lapl^* \psi = -\partial_i \big(a^{ij} \partial_j \psi + c^i \psi\big) + b^i \partial_i \psi + d\psi = 0$ with $\psi$ non-vanishing ($\Lapl^*$ is the adjoint). Then $v = \frac{\varphi}{\psi}$ satisfies
\begin{align*}
-\partial_i\big(\psi^2(a^{ij} \partial_j v + (b^i - c^i)v)\big) = 0
\end{align*}
\end{lemma}
\begin{proof}
This is just a lengthy computation similar to the proof of Lemma \ref{lem1}. See also \cite{Schulz} for a use of this identity; a more involved identity for $A$ non-symmetric can be found in \cite{ales}.
\end{proof}

We are now in shape to prove the following counterexample:

\begin{theorem}
Assume $\varepsilon > 0$, $f_1, f_2, f_3$ and $c, d$ are as in Theorem \ref{cexn=4}; switch the sign on $a$ and $b$ in the same Theorem. Then there exists a smooth $\psi > 0$, such that $g_k := \frac{f_k}{\psi}$ satisfy, for $k = 1, 2, 3$:
\begin{align*}
\Lapl' g_k = -\partial_1\big(\psi^2 (a\partial_1g_k + (c + \partial_1 a) g_k )\big) - \partial_2\big(\psi^2( b\partial_2 g_k + (d + \partial_2 b) g_k)\big) &= 0\\
\Lapl' \Big(\frac{1}{\psi}\Big) &= 0
\end{align*}
Therefore, $g_3 = g_1 g_2$ on $B_{\varepsilon}$ but not on $B_{2\varepsilon} \setminus B_\varepsilon$ and so we have a contradiction to the WUCP for the divergence type operator $\Lapl'$ and the matrix function $G := \begin{pmatrix}
g_3 & g_2\\
g_1 & \frac{1}{\psi}
\end{pmatrix}$
\end{theorem}
\begin{proof}
We rewrite the equations from Theorem \ref{cexn=4} in the following form:
\begin{align*}
0 = \Lapl f_k = -\partial_1 \big(a \partial_1 f_k\big) - \partial_2 \big(b \partial_2 f_k\big) + \big(c + \partial_1 a\big)\partial_1 f_k + \big(d + \partial_2 b\big) \partial_2 f_k
\end{align*}
We want to apply Lemma \ref{reductionlemma} to the operator $\Lapl^*$, or in other words we want to solve
\begin{align*}
\Lapl^* \psi = -\partial_1\big(a \partial_1 \psi + (c + \partial_1 a)\psi \big) - \partial_2 \big(b \partial_2 \psi + (d + \partial_2 b) \psi\big) = 0
\end{align*}
with $\psi > 0$. But we can just solve the Dirichlet problem for $\Lapl^* \psi = 0$ with $\psi = 1$ on $\partial B_{2\varepsilon}$; then the minimum principles for $\Lapl^*$ give that $\psi \geq 1$ in the whole of $B_{2\varepsilon}$. So we may apply the previous Lemma to get $g_k = \frac{f_k}{\psi}$ satisfying $\Lapl' g_k = 0$ for $k = 1, 2, 3$.

Furthermore, since $\Lapl (1) = 0$ we clearly have $\Lapl' \Big(\frac{1}{\psi}\Big) = 0$. The conclusion follows from the definition of $f_k$ for $k = 1, 2, 3$.
\end{proof}

\begin{rem}\rm
Note that the technique in Theorem \ref{cexn=4} cannot be applied to only coefficients next to first order and zero order derivatives, for example since in 2D we have $f_1 f_2 = f_3$ implies linear dependence of rows of first order derivatives, so a determinant would vanish. Therefore, we must use coefficients next to second order derivatives.

The question of whether there is a counterexample for the pure divergence operators of the form $-\partial_i(a^{ij}\partial_j \cdot)$ remains open.
\end{rem}

\section{Positive results}\label{sec3}
In Sections \ref{sec3}, \ref{sec4}, \ref{sec5} and \ref{sec6} we outline a few approaches to proving the SUCP or WUCP in Question \ref{mainque} in different situations. As we have seen previously, there is little hope in proving UCPs for general operators of form \eqref{eqn}, so we need to restrict the class we consider. In particular, we are interested in 
\begin{enumerate}
\item[1.] Divergence type operators $\partial_i(a^{ij}\partial_j)$.\label{type1div}
\item[2.] Conformally Euclidean metrics, i.e. operators of type 1. with $a^{ij}(x) = c(x)\delta^{ij}$ for some positive function $c(x)$.\label{type2confeucl}
\item[3.] Elliptic operators of the form $a^{ij}\partial_i \partial_j + b_i\partial_i + c$.\label{type3genell}
\end{enumerate}

Note that the Laplace-Beltrami operator given by $\Delta_g = -\frac{1}{\sqrt{|g|}} \partial_i \big(\sqrt{|g|} g^{ij} \partial_j\big)$ is of divergence type. 

In this section, we prove a positive result in the case 1. above with $n = 1$. The proof uses elementary properties of solutions to ODEs in 1D.




\begin{prop}[Divergence type for $n = 1$]\label{divn=1}
Let $m \in \mathbb{N}$. Assume $F: \mathbb{R} \to \mathbb{C}^{m \times m}$ is a smooth matrix function satisfying
\[\frac{d}{dt}\big(a \frac{dF}{dt}\big) = 0\]
for a positive smooth function $a$ on $\mathbb{R}$. If $\det F$ vanishes to order $(m + 1)$ at $0$, then $\det F = 0$ on the whole of $\mathbb{R}$. So both the SUCP and WUCP hold in this case.
\end{prop}
\begin{proof}
Note that for an entry $f$ of $F$, we have
\[\frac{df}{dt} = \frac{C(f)}{a}\]
where $C(f)$ is a constant. Therefore, if $\frac{df}{dt}$ vanishes at any point, we must have $f$ constant. If all entries of $F$ are constant, we are done. If we have $\frac{df}{dt} \neq 0$, then for any other entry $g$ of $F$, we have $\frac{dg}{dt} = C(f, g) \frac{df}{dt}$ for a constant $C(f, g)$ and consequently, we must have $g = C(f, g)f + C'(g)$ for another constant $C'(g)$.

Thus, there exists a holomorphic polynomial $p$ of degree up to $m$, such that $\det F (t) = p (f(t))$ for all $t$. Since $\frac{df}{dt} \neq 0$, $f$ maps $[-\varepsilon, \varepsilon]$ diffeomorphically to $f([-\varepsilon, \varepsilon]) \subset \mathbb{C}$ for some $\varepsilon > 0$, by the inverse function theorem.

By the chain rule, we obtain that $p$ vanishes to infinite order at $f(0)$, but since $p$ is a holomorphic polynomial, this is impossible unless $p \equiv 0$ and so $\det F \equiv 0$. 
\end{proof}

The proof of the above Proposition works for operators of the form $\frac{d^2}{dt^2} + a\frac{d}{dt}$ in the same way, but what about $P = \frac{d^2}{dt^2} + a\frac{d}{dt} + b$? The following Proposition answers our third question above positively.

\begin{prop}
Let $F: \mathbb{R} \to \mathbb{C}^{m \times m}$ be a smooth matrix function and we consider, for smooth $a$ and $b$
\[P = \frac{d^2}{dt^2} + a\frac{d}{dt} + b\]
Then $PF = 0$ and $\det F$ vanishing to order $(m + 1)$ at zero implies that $\det F \equiv 0$. So $\det F$ satisfies both the SUCP and WUCP.
\end{prop}
\begin{proof}
We follow the proof of Proposition \ref{divn=1}. In this case, the solution space to $Pf = 0$ is two dimensional, depending on values $f(0)$ and $\frac{df}{dt}(0)$. Say this is spanned by $f_1$ and $f_2$, where $f_1(0) = 1$ and $\frac{df_1}{dt}(0) = 0$, while $f_2(0) = 0$ and $\frac{df_2}{dt}(0) = 1$.

Since every entry is a linear combination of $f_1$ and $f_2$, we obtain that $\det F(t) = p\big(f_1(t), f_2(t)\big)$, where $p$ is a homogeneous holomorphic polynomial in two variables of degree $m$. But then using homogeneity we get $p\big(f_1(t), f_2(t)\big) = f_1^m(t) p\big(1, \frac{f_2(t)}{f_1(t)}\big)$ near zero, so the auxiliary polynomial $q(z) = p(1, z)$ vanishes to order $(m+1)$ at $z = 0$ and so $q \equiv 0$, implying $\det F \equiv 0$.
\end{proof}

Together with our counterexample Proposition \ref{cexn=1}, this circles up the story for $n = 1$.

\section{Harmonic conjugates} \label{sec4}
Here we focus mostly on the $m = 2$ and $n = 2$ case and operators of divergence form. Recall that two functions $u$ and $v$ on $\mathbb{C}$ are \emph{harmonic conjugate} if $u + iv$ is holomorphic. In other words, $u$ and $v$ satisfy the Cauchy-Riemann equations.

Given $\Omega \subset \mathbb{C}$ simply connected, then given a harmonic function $u$, there exists a unique harmonic conjugate function (up to constant), that is given by integrating the rotated gradient along an arbitrary curve; that this is well-defined follows from the divergence theorem.

More generally, given a smooth metric $g$ on $\Omega \subset \mathbb{C}$ simply connected, we say that two harmonic functions (i.e. $\Delta_g a = \Delta_g b = 0$) $a$ and $b$ are \emph{harmonic conjugate with respect to $g$} if $da = \star db$, where $\star$ is the Hodge star\footnote{In $2$ dimensions, the Hodge star $\star$ is just the rotation by $90$ degrees clockwise.}. Given just a harmonic function $b$, then $a$ exists and is unique up to constants. This follows from the fact that the Laplace-Beltrami operator can be written as $\Delta_g = d^*d$, where $d^* = \star d \star$ and $\star^2 = -1$ on one forms. The harmonicity of $b$ implies $\star db$ is closed, so $a$ exists and is unique up to constants. Moreover, this $a$ is clearly also harmonic and we also notice that $|da|_g = |db|_g$. 

Also, note that if given two harmonic function $a$ and $b$ with $\Delta_g a = \Delta_g b = 0$ with $\langle{da, db}\rangle = 0$ in $\Omega$, then $a$ and $b$ are harmonic conjugates w.r.t. $g$ (up to constants). To see this, note that $da = \lambda \star db$ for some function $\lambda$ and so by applying $d$ and $d\star$ to both sides, we deduce $d\lambda = 0$. This implies $\lambda$ is constant and so we get our conclusion.

Moreover, it is enough to have $\langle{da, db}\rangle = 0$ on an open subset $\Omega' \subset \Omega$ to conclude $a$ and $b$ are conjugate in $\Omega$: namely, note that $a$ determines a unique harmonic conjugate $b'$ in $\Omega$, which is by previous paragraph equal to $b$ in $\Omega'$ (up to multiplication by constant). Thus, by WUCP for $\Delta_g$, we get $b \equiv b'$ is conjugate to $a$ on the whole of $\Omega$.

How can we extend this to an arbitrary operator of divergence type $P = \partial_i(a^{ij} \partial_j)$, where $a^{ij}$? Notice firstly that for the operator $\Delta_g$ we have the corresponding $a^{ij} = \frac{g^{ij}}{\sqrt{|g|}}$ where $|g| = \det g$ and in this case $\det A = 1$, where $A_{ij} = a^{ij}$ is the associated matrix. See the next section for the proper treatment of the case of general $A$ and the corresponding structures.

We first present a useful Lemma producing an equation for the quotient of the two solutions.



\begin{lemma}\label{lem1}
Let $f$ and $g$ be two smooth functions in $\mathbb{R}^n$ with $Pf = Pg = 0$ and $g \neq 0$, then $\partial_i(g^2 a^{ij} \partial_j \frac{f}{g}) = 0$, so in other words $\frac{f}{g}$ also satisfies a divergence type equation.
\end{lemma}
\begin{proof}
This follows easily by computation: 
\begin{align*}
0 = \partial_i(a^{ij}\partial_jf) = \partial_i\big(a^{ij}\partial_j \big(g \cdot \frac{f}{g}\big)\big) = Pg \cdot \frac{f}{g} + 2a^{ij}\partial_i g \partial_j \big(\frac{f}{g}\big) + g \cdot P\big(\frac{f}{g}\big)
\end{align*}
we multiply both sides with $g$, use chain rule and $Pg = 0$ to re-write this as:
\begin{align*}
0 = a^{ij}\partial_i (g^2) \partial_j \big(\frac{f}{g}\big) + g^2 \cdot P\big(\frac{f}{g}\big) = \partial_i\big(g^2 a^{ij} \partial_j \big(\frac{f}{g}\big)\big)
\end{align*}
\end{proof}

Note that if $m = 2$, then this makes us able to reduce the problem (locally) to the case where $F = \begin{pmatrix}
h & g\\
f & 1
\end{pmatrix}$
by dividing with a non-zero entry and using Lemma \ref{lem1} to reduce the problem to a matrix of this form, by redefining $A$. Observe now that if $Pf = Pg = 0$, then $P(fg) = 0$ if and only if $a^{ij}\partial_i f \partial_j g = 0$, i.e. $df$ and $dg$ are orthogonal w.r.t. $A$.

\begin{rem}\rm
If $\det A$ is constant and $A$ is symmetric, then by our discussion above, if $\det F = 0$ in a neighbourhood $\Omega'$ of the origin, then $df$ and $dg$ are orthogonal w.r.t. $A$ in $\Omega'$ and so there is a unique harmonic conjugate (up to constants) to $f$ in $\Omega$; so by unique continuation $g$ is the harmonic conjugate (up to constants) in $\Omega$, too. So we prove the WUCP in this case.

For the proof of the SUCP in this case or in other words, of the fact that $\Delta_g a = \Delta_g b = \Delta_g c = 0$ with $c - ab = O (|x|^\infty)$ at zero implies $c \equiv ab$, see Proposition \ref{SUCP1}.
\end{rem}

Recall the existence of \emph{harmonic coordinates} for surfaces. These are tied with the harmonic conjugates: given $(\Omega, g) \subset \mathbb{R}^2$ and a point $p \in \Omega$, one builds an harmonic function $u$ with $\Delta_g u = 0$ and $u(p) = 0$ with $\nabla u (p) \neq 0$. Then by parametrising with $u$ and the harmonic conjugate of $u$ we get \emph{isothermal coordinates} in which $g = \begin{pmatrix}
\lambda & 0\\
0 & \lambda
\end{pmatrix}$
for a positive function $\lambda$. Note that due to conformal invariance, the harmonic function $h$ in these coordinates satisfies
\begin{align}\label{eqneucl}
\Delta_{eucl} h = \Big(\frac{\partial^2}{\partial x^2} + \frac{\partial^2}{\partial y^2}\Big)h = 0
\end{align}
and so is harmonic in the usual sense. In particular we have
\begin{theorem}\label{isothermal2dsucp}
Let $(\Omega, g)$ be a planar domain with $g$ of class $C^{1, \alpha}$ for $\alpha > 0$. Then let $p: \mathbb{C}^n \to \mathbb{C}$ be a real analytic function. If $\Delta_g f_i = 0$ for $f_i \in C^{2, \alpha}$ and $i = 1, \dotso, N$ for $N \in \mathbb{N}$ and moreover, if $p(f_1, \dotso, f_N)$ vanishes to infinite order at zero, then $p(f_1, \dotso, f_N) \equiv 0$.

In particular, we may choose $p(F) = \det(F)$ to be the determinant of an $\mathbb{C}^{m \times m}$ matrix function and so in this case we have the SUCP.
\end{theorem}
\begin{proof}
In these conditions, there exist isothermal coordinates \cite{DTK} (c.f. previous paragraph) and in these coordinates $g \in C^{2, \alpha}$. Moreover, we see that $f_i$ satisfy \eqref{eqneucl}, i.e. they are harmonic in the new coordinates. Therefore by elliptic regularity they are smooth and moreover, analytic. So the composition $p(f_1, \dotso, f_N)$ is analytic and vanishes to infinite order and so must entirely vanish.
\end{proof}

\begin{rem}\rm
We might object and say that the previous proof relies on the analyticity. Is there a proof of SUCP for the determinant that does not use analyticity? We sketch this as follows. We note that if $\Delta_g a = \Delta_g b = \Delta_g c = 0$ for $(\Omega, g) \subset \mathbb{R}^2$, then $c - ab = O(|x|^\infty)$ implies $\langle{da, db}\rangle_g = O(|x|^\infty)$. This orthogonality relation can be seen to determine the full jet of $b$ at zero (up to constants) by going to isothermal coordinates, in which the Taylor polynomials of $a$ and $b$ of any order are harmonic. Then one may inductively determine the Taylor coefficients of $b$ from $a$ and the metric, by using this harmonicity of the coefficient polynomials. This implies that $b$ has the same Taylor expansion as the harmonic conjugate of $a$ and so by the SUCP for $\Delta_g$, we see that $b$ must be the harmonic conjugate of $a$.

For a different proof, see Propositon \ref{SUCP1}.
\end{rem}

We continue our study of the 2D case in divergence form by looking at the blow ups of solutions at a point. More precisely, we look at the leading terms of Taylor polynomials of solutions to equations of elliptic operators. Then we have
\begin{prop}
Let $u$ be a smooth solution to $\Lapl u = 0$ in $\mathbb{R}^n$ for any $n$, where $\Lapl$ is any one of the three classes of operators in \eqref{type1div}. Then after a linear change of coordinates, the top Taylor coefficient at zero is harmonic.
\end{prop}
\begin{proof}
Change the coordinates by a linear transformation such that the principal part at zero is just $\sum \partial_i^2$. Assume the order of vanishing at zero of $u$ is $N$. Let us introduce $u_r(x) := r^{-N}u(rx)$. Then by Taylor's theorem, $u_r \to p_N$ locally uniformly as $r \to 0$ (with all derivatives), where $p_N$ is the $N$-th Taylor polynomial of $u$. Note that $u_r$ satisfies the following equation:
\begin{align}\label{eqnscaled}
a_{ij}^r \partial_{ij} u_r + b_i^r \partial_i u_r + c^ru_r = 0
\end{align}
Here $a_{ij}^r(x) = a_{ij}(rx)$, $b_i^r(x) = r b_i(rx)$ and $c^r(x) = r^2c(rx)$. Note that we have $a_{ij}^r \to a_{ij}(0)$, $b_i^r \to 0$ and $c^r \to 0$ locally uniformly as $r \to 0$, so when the limit is taken we get
\[\sum \partial_i^2 p_N = 0\]
\end{proof}

\begin{rem}\rm
The above Proposition can be generalised to less smooth coefficients $a_{ij}, b_i, c \in C^2_{loc}(\mathbb{R}^n)$ by considering the order of vanishing of a function $u \in L^2_{loc}(\mathbb{R}^n)$ -- the least non-negative integer $N$ such that there	exists $R > 0$ and constants $c_1, \dotso, c_N$
\[\int_{B(0, r)} |u(x)|^2 dx \leq c_k^2 r^{2k + n}\]
for all $r \leq R$ and $1 \leq k \leq N$ (see \cite{KT07}). 

Then with $u_r(x) = r^{-N}u(rx)$ as before and $u \in H^2_{loc}(\mathbb{R}^n)$ satisfying $\Lapl u = 0$, we have that $u_r$ is bounded uniformly as $r \to 0$ in $H^3\big(B(0, 1)\big)$ by the scaled elliptic estimates $\lVert{u}\rVert_{H^3(B(0, r))} \lesssim \frac{1}{r^3} \lVert{u}\rVert_{L^2(B(0, r))}$ (note that $D^3 u_r(x) = r^{3 - N} D^3 u(rx)$). So by Rellich compactness, we get a convergent subsequence in $H^{2} (B(0, 1))$ and by taking the $r_k \to 0$ over this subsequence in \eqref{eqnscaled}, that $u_r$ in the limit is harmonic. Note we could have applied the same argument for coefficients in $C^{1, \alpha}$ for any $\alpha > 0$; also, the $L^2$ norm could be replaced by the $\sup$ norm in the above definition of the order of vanishing, by use of Schauder estimates and Arzela-Ascoli.
\end{rem}

This takes us to proving the following claim, which is an elementary result classifying pairs of harmonic polynomials satisfying a certain property.

\begin{lemma}\label{harmpoly}
Assume we have four non-zero, real harmonic, homogeneous polynomials $p_{ij}$ in $\mathbb{R}^2$ for $i, j = 1, 2$ with $p_{11}p_{22} = p_{12} p_{21}$. Then one of the following two holds, up to constants and permutations:
\begin{itemize}
\item We are in the trivial case, $p_{11} = p_{12}$ and $p_{22} = p_{21}$.
\item We have $p_{22} = 1$, $p_{12} = A \re (z^k) + B \im (z^k)$ and $p_{12} = C \re (z^k) + D \im (z^k)$ for $A, B, C, D \in \mathbb{R}$ with $AC + BD = 0$ and $k \in \mathbb{N}$. Of course, then $p_{11} = p_{12} p_{21}$.
\end{itemize}
Conversely, in any of the two cases we get a quadruple of harmonic polynomials with $p_{11}p_{22} = p_{12} p_{21}$.
\end{lemma}

For a proof, see the Appendix \ref{secapp}. We combine this Lemma with Lemma \ref{lem1} to reduce the problem to the case where one entry is equal to one.

\begin{prop}\label{reductionf22=1}
Assume $f_{ij}$ are smooth and $A$-harmonic\footnote{$u$ is $A$-harmonic if $\div A\nabla u = 0$. See Section \ref{sec5} for more details} for $i, j = 1, 2$ and satisfy $f_{11}f_{22} - f_{12}f_{21} = O(|x|^\infty)$ at zero. Then if $f_{ij}$ all vanish at zero, we must have $f_{11}f_{22} = f_{12} f_{21}$ on the whole domain.

Consequently, by Lemma \ref{lem1} we reduce the problem to the case where one entry is equal to $1$.
\end{prop}
\begin{proof}
We can assume that the matrix $A$ is the identity at zero by a linear change of coordinates. Then the leading Taylor polynomials $p_{ij}$ of $f_{ij}$ are harmonic and satisfy $p_{11}p_{22} = p_{12}p_{21}$ by the condition on $f_{ij}$. If one of the entries vanishes to infinite order, then by the usual SUCP it is zero throughout and we easily conclude $f_{11}f_{22} = f_{12} f_{21}$ on the whole domain, after another use of the SUCP. 

By Lemma \ref{harmpoly} and since $f_{ij}$ all vanish at zero, we know we are in the second case; i.e. up to constants and permutations we may assume $p_{11} = p_{12}$ of degree $r > 0$ and $p_{22} = p_{21}$ of degree $s > 0$. We distinguish two cases: $r > s$ ($r < s$ is analogous) and $r = s$.

If $r > s$, then by subtracting the second column from the first column (i.e. after the linear transform $f_{11} \mapsto f_{11}' = f_{11} - f_{12}$ of degree $r'$ and $f_{21} \mapsto f_{21}' = f_{21} - f_{22}$ of degree $s'$), we increase the orders of vanishing of the first column, i.e. $r' > r$ and $s' > s$. Moreover, the determinant is unchanged and we notice that $r' > r > s$, which gives a contradiction (unless, $r'$ or $s'$ are equal to $\infty$, which we know how to deal with).

If $r = s$, by the same subtraction procedure we may reduce to the case where we have $r > s$. This finishes the proof of the first claim.

Finally, for the second claim note that if we have $f_{22}(0) \neq 0$, then by Lemma \ref{lem1} we may assume that locally $f_{22} \equiv 1$.
\end{proof}

Note that by Lemma \ref{isothermal2dsucp} we know how to solve the $\det A = 1$ case. The following Proposition tells us that if $u, v$ and $w$ satisfy $Pu = Pv = Pw = 0$ and $w - uv = O(|x|^\infty)$, then $v$ is the harmonic conjugate of $u$ up to constants -- but we do not use analyticity.

\begin{prop}\label{SUCP1}
Assume $u, v$ and $w$ are smooth (real or complex) and satisfy $Pu = Pv = Pw = 0$ with $\det A = 1$. Then $w - uv = O(|x|^\infty)$ implies $w = uv$ on the whole domain and that $v$ is the harmonic conjugate of $u$.
\end{prop}
\begin{proof}
We first consider the case where $dv(0) \neq 0$. Then we may write
\begin{align}\label{eqnlambdamu}
du = \lambda \star dv + \mu dv
\end{align}
for some functions $\mu$ and $\lambda$. The condition $w - uv = O(|x^\infty|)$ implies that $\langle{du, dv}\rangle_A = O(|x|^\infty)$ ($A$ corresponds to a Riemannian metric) and so $\mu = O(|x|^\infty)$. By applying $d$ and $d\star$ do this equation respectively, we get
\begin{align}
d\lambda \wedge \star dv = O(|x|^\infty) \quad \text{and} \quad d\lambda \wedge dv = O(|x|^\infty)
\end{align}
which in turn implies $\lambda = \lambda(0) + O(|x|^\infty)$. Therefore
\begin{align}
du = \lambda(0) \star dv + O(|x|^\infty)
\end{align}
But there is the harmonic conjugate $u'$ to $u$, so that $d(u - \lambda(0)u') = O(|x|^\infty)$ and so by the usual SUCP we get $u - \lambda(0)u'$ is constant, which finishes the proof.

If $dv(0) = 0$, then by assuming $A(0) = Id$ we may argue by the second case of Lemma \ref{harmpoly} to get that $\lambda$ and $\mu$ extend to zero smoothly, by Taylor's theorem (note also that the zeros of $dv$ are isolated if $v$ is non-constant\footnote{This is true by e.g. going to coordinate system given by Lemma \ref{isotropicoord}, reducing the problem to a first order equation for $\partial v$ and using the results of \cite{bar}}). Once we have equation \eqref{eqnlambdamu}, we argue in the same manner.
\end{proof}

The problem of generalising the above Proposition is that if $\det A \neq 1$, then the harmonic conjugate is $A^*$-harmonic and $A^* \neq A$ in general (see the next section for the definition of these concepts). In the next proposition, we reduce the problem to the \emph{isotropic} case, i.e. the case of $A = \lambda \times Id$ for positive $\lambda$.

\begin{prop}\label{mainreduction}
In proving the SUCP for the determinant and operators of divergence type where $A$ is symmetric, it is enough to consider the isotropic case.

By combining with Proposition \ref{reductionf22=1}, we are also reduced to the case where $f_{22} = 1$.
\end{prop}
\begin{proof}
Given a symmetric $A$, we have by Lemma \ref{isotropicoord} a diffeomorphism $F$ such that $F_*A = \tilde{a} Id$ for a positive function $\tilde{a}$ (here $F_*$ is the pushforward). This finishes the proof. 
\end{proof}

\begin{rem}\rm
Note that we do not need to have $\det A$ constant always, if $u, v$ and $w$ satisfy $P u = Pv = Pw = 0$ and $w = uv$, or $v$ to be conjugate to $u$. For example, we may take $A = \begin{pmatrix}
1 & 0\\
0 & a
\end{pmatrix}$ with $a(x, y) = \frac{f(x)}{g(y)}$ with $f$ and $g$ positive, and let $u(x, y) = x$, $v(x, y) = v(0) + \int_0^y g(t)dt$. Then $uv$ is also $A$-harmonic and we also have $\det A = \frac{f}{g}$ which is not constant for general $f$ and $g$. Moreover, we easily check that $y$ is the harmonic conjugate to $x$, so also in general $v$ is not the harmonic conjugate to $u$. 

It is tempting to say that we will have $w' = u'v'$, but this is also false: let $u = x$, $v = y$ and $w = xy$ for $a = 1$ as above. Then $u' = y$, $v' = -x$ and $w' = \frac{1}{2}(-x^2 + y^2)$, so $w' \neq u'v'$.
\end{rem}

\section{More general operators of divergence type}\label{sec5}

Following \cite{astala} (Chapter 16.) we consider the case of divergence type where $\det A$ is not necessarily constant or $A$ is not symmetric, by relating the study of elliptic equations in 2D to complex analysis. The main conclusions of this section are reduction results, i.e. we prove it is sufficient to consider special forms of $A$. We assume $A$ is bounded and strongly elliptic on $\Omega \subset \mathbb{C}$, i.e. there exists $K > 0$ such that
\begin{align}\label{strongellipticity}
\frac{1}{K}|\xi|^2 \leq \langle{A(z) \xi, \xi}\rangle \leq K |\xi|^2
\end{align}
for a.e. $z \in \Omega$ and all $\xi \in \mathbb{R}^2$. We call a function $u$ \emph{$A$-harmonic} if 
\[\div \big(A \nabla u\big) = 0\]
where we assume $A$ is just positive definite. This motivates the definition of a harmonic conjugate function $v$ to $u$:
\[\nabla v = J A \nabla u\]
Here $v$ exists and is uniquely determined up to constant. Note that $v$ is $A^*$-harmonic, where $A^* = -JA^{-1}J = \frac{A^T}{\det A}$, i.e.
\[\div \big(A^* \nabla v\big) = 0\]
Now the relation to complex analysis is yielded by defining $f = u + iv$ and noting that $f$ satisfies a \emph{Beltrami type equation}:
\begin{align}\label{beltrami}
\Lapl f = \frac{\partial f}{\partial \bar{z}} - \mu(z) \frac{\partial f}{\partial z} - \nu(z) \overline{\frac{\partial f}{\partial z}} = 0
\end{align}
where $\mu$ and $\nu$ depend only on $A$. Note that when $A = Id$, then $\mu = \nu = 0$ and we obtain the Cauchy-Riemann equations. The following Lemma (Theorem 16.1.6. of \cite{astala}) states precisely this connection:
\begin{lemma}\label{complexrep}
Let $\Omega$ be a simply connected domain and let $u \in W^{1, 1}_{loc}(\Omega)$ be a solution to
\[\div(A \nabla u) = 0\]
If $v \in W^{1,1}_{loc}(\Omega)$ is the harmonic conjugate to $u$ and $f = u + iv$ satisfies \eqref{beltrami} with:
\begin{align}\label{system1}
\mu &= \frac{1}{\det(I + A)}\big(A_{22} - A_{11} - i(A_{12} + A_{21})\big)\\
\nu &= \frac{1}{\det(I + A)}\big(1 - \det A + i(A_{12} - A_{21})\big)\label{system2}
\end{align}
Conversely, if $f \in W^{1, 1}(\Omega)$ satisfies \eqref{system1} and \eqref{system2}, then $u = \re(f)$ is $A$-harmonic and $v = \im(f)$ the harmonic conjugate of $u$. 
\end{lemma}
There are also formulas expressing the entries of $A$ in terms of $\mu$ and $\nu$,  but we do not need them here. Note just that $A$ is symmetric if and only if $\nu$ is real valued and that $\det A = 1$ if and only if $\nu$ is pure imaginary; so $A$ is symmetric and has $\det A = 1$ if and only if $\nu = 0$.

Another ingredient we will need is a version of \emph{Stoilow factorisation} for operators of the form \eqref{beltrami}. The statement in general is that every $K$-quasiregular map factorizes as a composition of a harmonic map and a quasiconformal homeomorphism. Here, a homeomorphism $f: \Omega \to \Omega'$ in $W_{loc}^{1,2}$ is \emph{$K$-quasiconformal} if and only if $\frac{\partial f}{\partial \bar{z}} = \mu(z) \frac{\partial f}{\partial z}$ for almost every $z \in \Omega$, where $\lVert{\mu}\rVert_{\infty} \leq \frac{K - 1}{K + 1}$.\footnote{So in particular, $f$ is $1$-quasiconformal if and only if it is conformal, i.e. holomorphic and injective.} Moreover, a mapping $f$ is \emph{$K$-quasiregular} if all hypothesis hold as above, except that we do not ask that $f$ is a homeomorphism.\footnote{For instance, this result shows a few nice things about quasiregular maps: they are open and discrete, local $\frac{1}{K}$-H\"older, differentiable with non-vanishing Jacobian a.e..}

More precisely, we will need the following form of Stoilow factorisation for general elliptic systems (Theorem 6.1.1. in \cite{astala}):
\begin{theorem}\label{Stoilow}
Let $f \in W^{1,2}_{loc}(\Omega)$ be a homeomorphic solution to \eqref{beltrami}, where we assume $|\mu| + |\nu(z)| \leq k <1$. Then any other solution $g \in W^{1, 2}(\Omega)$ to $\Lapl g = 0$ takes the form $g = F\big(f(z)\big)$, where $F$ is a $K^2$-quasiconformal mapping satisfying
\begin{align}\label{reducedbeltrami}
\frac{\partial F}{\partial \bar{w}} = \lambda(w) \im\Big(\frac{\partial F}{\partial \bar{w}}\Big)
\end{align}
for $w \in f(\Omega)$, where (here $z = f^{-1}(w)$)
\begin{align*}
\lambda(w) = \frac{-2i \nu(z)}{1 + |\nu(z)|^2 - |\mu(z)|^2}
\end{align*}
It is easily seen that $|\lambda| \leq \frac{2k}{k^2 + 1} < 1$. Conversely, for any such $F \in W^{1, 2}(\Omega)$ satisfying \eqref{reducedbeltrami}, $g = F \circ f$ solves $\Lapl g = 0$.
\end{theorem}
We call the equation \eqref{reducedbeltrami} the \emph{reduced Beltrami equation}. We need these two results for the following:

\begin{lemma}[Variant of the isothermal coordinates]\label{genisothermal}
Let $A$ be smooth and strongly elliptic, i.e. satisfying \eqref{strongellipticity}. For any $p \in \Omega$, there exists a $C^\infty$ coordinate chart $\varphi: p \in \Omega' \to \mathbb{C}$, such that for any solution $u$ to
\[\div \big(A \nabla u\big) = 0\]
can be written as $u = v \circ \varphi$, where $v$ satisfies $\div \big(\widetilde{A} \nabla v\big) = 0$ with $\widetilde{A} = \begin{pmatrix}
1 & \widetilde{A}_{12}\\
0 & \widetilde{A}_{22}
\end{pmatrix}$, where (for $\varphi(z) = w$)
\begin{align}\label{eqnB12}
\widetilde{A}_{12}(w) = \frac{-2\im(\lambda)(w)}{1 - \re(\lambda)(w)} \quad \text{ and } \quad \widetilde{A}_{22}(w) = \frac{1 + \re(\lambda)(w)}{1 - \re(\lambda)(w)}
\end{align}
Here $\lambda(w)$ is given by \eqref{lambda}, where we insert $\mu(z)$ and $\nu(z)$ from the equations \eqref{system1} and \eqref{system2}.

Moreover, if $A$ is symmetric then $\widetilde{A}_{22} = 1$; if $\det A = 1$, then $\widetilde{A}_{12} = 0$.
\end{lemma}
\begin{proof}
This is clear by combining Lemma \ref{complexrep} and Theorem \ref{Stoilow}. Consider the harmonic conjugate $u'$ of $u$ and $f = u + iu'$. As a first step, similarly to the proof of existence of isothermal coordinates (which are a special case)\footnote{By taking $A$ symmetric and with $\det A = 1$, we recover the isothermal charts.}, we take an $A$-harmonic function $u_1$ with $u_1(p) = 0$ and $\nabla u_1(p) \neq 0$. Then by taking the harmonic conjugate of $u_1'$, we get a coordinate system locally and define $f_1 = u_1 + iu_1'$, which is a local homeomorphism such that $F: = f \circ f_1^{-1}$ satisfies the reduced Beltrami equation \eqref{reducedbeltrami}. By noting that
\begin{align}\label{lambda}
\im \Big(\frac{\partial}{\partial z}\Big) = \frac{1}{2}\Big(\frac{\partial}{\partial z} - \overline{\frac{\partial}{\partial z}}\Big)
\end{align}
we have $\tilde{\mu}(w) = -\tilde{\nu}(w) = \frac{\lambda(w)}{2}$ in these new coordinates, where $\lambda(w)$ is given by \eqref{lambda}.

By comparing the coefficients of the new matrix $\widetilde{A}$ in the equations \eqref{system1} and \eqref{system2} we get $\widetilde{A}_{21} = 0$ and $\widetilde{A}_{22} - \widetilde{A}_{11} + 1 - \widetilde{A}_{11}\widetilde{A}_{22} = 0$, which makes us able to assume $\widetilde{A}_{11} = 1$. Then it is easy to get \eqref{eqnB12} by taking the real and imaginary parts of \eqref{system1} for example.

Finally, from \eqref{system2} we know that $\nu$ is real if and only if $A$ is symmetric; $\nu$ is pure imaginary if and only if $A$ has $\det A = 1$. The last claim now follows from equation \eqref{lambda}. 
\end{proof}

We separately state a result in the same vein as the previous Lemma; it gives a coordinate system such that $A$ is isotropic. The proof is similar as for the previous two results. For a proof, see the proof of Lemma 3.1. in \cite{ALP} and references therein.
 
\begin{lemma}\label{isotropicoord}
Assume $A$ is symmetric. Given a point $p \in \Omega$, there exists a local diffeomorphism such that $F_*A = \tilde{a} \times Id$, where $\tilde{a}(z) = \det(A(F^{-1}(z)))^{\frac{1}{2}}$. Here $F_*$ denotes the pushforward and $F$ is a solution to the Beltrami equation
\begin{align*}
\frac{\partial F}{\partial \bar{z}} = \mu(z) \frac{\partial F}{\partial z}
\end{align*} 
Here $\mu$ is determined explicitly by $A$ and is given by
\begin{align*}
\mu(z) = \frac{g_{11}(z) - g_{22}(z) + 2ig_{12}(z)}{2 + g_{11}(z) + g_{22}(z)}
\end{align*}
where $g_{ij}$ are the entries of the matrix $G = \sqrt{\det A} A^{-1}$.
\end{lemma}

\subsection{Non-self adjoint equations.}\label{reductions} We remark that by the methods of G. Alessandrini \cite{ales}, where he proves the SUCP properties for possibly non-self adjoint elliptic operators of divergence type with lower order coefficients, we may reduce the case of more general linear equations to an equation of the divergence type. It is based on a reduction method as in Lemma \ref{reductionlemma} and Lemma \ref{lem1}.
In fact, Alessandrini shows for possibly non-symmetric $A$, that we may introduce two positive multipliers $m, w$, such that the equation
\begin{align}
Lu = -\div(A\nabla u + u B) + C\nabla u + du
\end{align}
reduces to a simpler equation, in the following sense. Here $A$ is $2 \times 2$ matrix, $B$ and $C$ are vector functions and $d$ a function. We have for any $v$
\begin{align}
\widehat{L}v = w L(mv)
\end{align}
where $\widehat{L}u = -\div(\widehat{A}\nabla u + u \widehat{B})$. Again, this provides a reduction procedure for our problem and makes it sufficient to consider operators of the form $\widehat{L}$.

\section{The case $n = 2$ for the twisted Laplacian}\label{sec6}

Here we prove the SUCP for a special class of matrix operators on $\mathbb{R}^2$ which satisfy an additional equation; namely, we consider connections Laplacians of the form $P = d_A^*d_A$ for $A$ a connection, i.e. a matrix of one forms, where we assume the Yang-Mills equation \eqref{YMeqn} for $A$. The motivation is explained in the introduction.

\begin{lemma}\label{YMlemma}
Let $(\Omega, g) \subset \mathbb{R}^2$ be a domain equipped with a smooth metric $g$. Equip $\Omega \times \mathbb{C}^m$ for $m \in \mathbb{N}$ with a Yang-Mills connection $A$\footnote{Recall that $A$ is \emph{Yang-Mills} if $D_A^*F_A = 0$; here $D_A$ is the natural induced connection on the endomorphism bundle and $F_A = dA + A \wedge A$ is the curvature. See also the introduction.}. Assume $F \in C^\infty(\Omega, \mathbb{C}^{m \times m})$ satisfy $d_A^*d_A F = 0$. Then $\det F$ satisfies the SUCP, and so the WUCP.
\end{lemma}
\begin{proof}
Assume w.l.o.g. $\det F$ vanishes to infinite order at zero. As in Proposition \ref{isothermal2dsucp}, we look at isothermal coordinates near zero, so that $g = \begin{pmatrix}
\lambda & 0\\
0 & \lambda
\end{pmatrix}$ in these coordinates for a smooth, positive function $\lambda$. The Yang-Mills equations take the form
\begin{align}\label{YM1}
0 = D_A^* F_A = \star D_A \star (dA + A \wedge A)
\end{align}
Let us write simply $A = A_1 dx + A_2 dy$ for $A_1, A_2$ smooth $m \times m$ matrices. Then the above equation takes the form
\begin{align}\label{YM2}
0 = d^*(dA + A \wedge A) + \star [A, \star (dA + A \wedge A)]
\end{align}
where the second term can be rewritten as
\begin{align}
\frac{1}{\lambda} \big(-[A_1, G(A)]dy + [A_2, G(A)]dx\big)
\end{align}
where $G(A) = \lambda \star F_A$ is just a function of $A$. Note that we have, in isothermal coordinates:
\begin{align}\label{hodgeisothermal}
\star dx = - g^{11} |g|^{1/2} dy= -dy, \quad \star dy = dx \text{\, and \,} \star(|g|^{\frac{1}{2}} dx \wedge dy) = 1
\end{align}
Therefore, since $d^* = \star d \star$ and by \eqref{hodgeisothermal}, we have that the Yang-Mills equation \eqref{YM1} is of the following form: $\frac{1}{\lambda}$ times an expression depending only on $A$.

Now we have two choices. By taking the Coulomb gauge in which $d^*A = 0$ (see \cite{cek1}), we have that this condition is equivalent to:
\begin{align}
\frac{\partial A_1}{\partial x} + \frac{\partial A_2}{\partial y} = 0
\end{align}
By applying $d$ to this equation and adding to \eqref{YM2} (after multiplying with $\lambda$), we get an equation of the form
\begin{align}\label{YMnice}
\Delta_{eucl} A + Q(A, \nabla A) = 0
\end{align}
where $Q$ is an analytic (polynomial) function of its entries and $\Delta_{eucl}$ is the Euclidean Laplacian that acts diagonally. Therefore by a well-known property of elliptic equations, we have $A$ is analytic in this gauge. Furthermore, since $d_A^*d_A$ is equal to $\frac{1}{\lambda} P_A$, where $P_A$ is a second order elliptic operator depending only on $A$, we have that $F$ is also analytic in this gauge and so is $\det F$, implying the SUCP and WUCP.

Alternatively, we may consider the harmonic gauge for the connection, i.e. $d^*A = \frac{1}{\lambda}(A_1^2 + A_2^2)$ (see \cite{cek1} for more details). In this gauge, $A$ satisfies:
\begin{align}
\frac{\partial A_1}{\partial x} + \frac{\partial A_2}{\partial y} + A_1^2 + A_2^2= 0
\end{align}
As before, by applying $d$ to this equation and adding to \eqref{YM1} after multiplication by $\lambda$, we are back to the form of equation \eqref{YMnice} and hence to the previous case.
\end{proof}

\section{Applications to the Calder\'on problem for connections}\label{sec7}

Here we apply the result and the proof of Lemma \ref{YMlemma} to the Calder\'on problem for connections (see \cite{cek1, cek2, AGTU}), by using the technique of the proof of Theorem 1.2 from \cite{cek1} to produce a result for surfaces and bundles of \emph{arbitrary} rank. Calder\'on's problem is an inverse boundary value problem that has picked up a lot of attention in the past thirty and more years \cite{Usurvey}.

A similar result for connections was proved in \cite{cek1} for either rank one case and smooth metric, or arbitrary rank but \emph{analytic} metrics and the main novelty here is to extend these methods to the smooth $2$-dimensional case and arbitrary rank.

First, we have the following simple geometric lemma:

\begin{lemma}\label{geomlemma}
Let $(\Omega, g) \subset \mathbb{R}^2$ containing $0$ with $g$ smooth. Fix a smooth embedded curve $0 \in \gamma \subset \Omega$. Then the Riemannian distance function $f(q) := d^2(q, \gamma)$ from a point $q \in \Omega$ to $\gamma$, has the following Taylor expansion at $0$, for $q = (x, y)$:
\begin{align}
f(x, y) = \begin{pmatrix}
x & y
\end{pmatrix} P^Tg(0)P\begin{pmatrix}
x\\
y
\end{pmatrix} + O(|x|^3)
\end{align}
where $P$ is the projection to $\star \dot{\gamma}(0)$ along $\dot{\gamma}(0)$, where $\dot{\gamma}(0)$ is the unit tangent vector to $\gamma$ at $0$.
\end{lemma}
\begin{proof}
See Appendix \ref{secapp}.
\end{proof}

We are now ready to prove the main result of this section, Theorem \ref{mainthm}.

\begin{proof}[Proof of Theorem \ref{cexn=4}]
The proof is analogous to the proof of Theorem 1.2 from \cite{cek1}, once we have Lemma \ref{YMlemma}. Let us recall the proof briefly and underline the differences. Let $F$ and $G$ be $m \times m$ matrix functions solving $d_A^*d_A F = d_B^*d_B G = 0$ with $F = G$ on $\partial M$ and $F = G = Id$ on an open, non-empty set $V \subset \Gamma$. By Lemma \ref{YMlemma}, we have that the zero sets of $\det F$ and $\det G$ are covered by a countable union of curves $\{C_i \mid i \in \mathbb{N}\}$; by SUCP we have $H = FG^{-1}$ satisfying $H^*A = B$ in a neighbourhood of $V$ with $H$ unitary. 

Next, we perform the drilling procedure from \cite{cek1}. Near a point $p \in C_i$ where $\det F$ vanishes to order $k-1$ locally on $C_i$, meaning that $\det G = y^k g_1$ in the normal coordinate system to $C_i$, by Taylor's theorem, with $g_1(p) \neq 0$. We assume that for $y>0$ (locally) we have $H^*A = B$. Then
\begin{align}\label{Hextension}
H = FG^{-1} = \frac{F \adj G}{y^k g_1}
\end{align}
Notice that $H$ is smooth and bounded for $y > 0$, so by Taylor's theorem $F \adj G = y^k H_1$ for some smooth $H_1$ and so $H = \frac{H_1}{g_1}$ extends smoothly to $y < 0$.

Now, there exists smooth, invertible and unitary $X$ and $Y$, such that $A' = X^*A$ and $B' = Y^*B$ satisfy the Coulomb gauge equation. If we change coordinates to isothermal coordinates by a diffeomorphism $\varphi$ (with $\varphi(x, y) = (u, v)$), then $A'$ and $B'$ are analytic by Lemma \ref{YMlemma}. Moreover $H' := F' G'^{-1} = X^{-1}HY$ smoothly extends to $y < 0$, too. \emph{If} we had $H'$ analytic, then by $H'^*A' = B'$ for $\varphi^{-1}_2(u, v) > 0$ we would have $H'^* A' = B'$ on the whole chart by analyticity and so $H^*A = B$ for $y < 0$. What follows is the proof of this analyticity.

The main issue is that in the version of \eqref{Hextension} for $H'$, the distance function $y$ is not always analytic, since $g$ is just smooth. To work around this, go to isothermal coordinates via $\varphi$ and write
\begin{align}\label{F'G'}
F'(q) \adj (G')(q) = H_1'(q) \big(d(q, \gamma)\big)^k =  H_1'(q) \Big(\frac{d^2(q, \gamma)}{d^2_{eucl}(q, \gamma)}\Big)^\frac{k}{2} \big(d_{eucl}^2(q, \gamma)\big)^\frac{k}{2}
\end{align}
where $\gamma = \varphi(C_i)$, $d_{eucl}$ is the Euclidean distance and $d(q, \gamma)$ denotes the distance of the point $q$ in the chart from $\gamma$ (w.r.t. the isothermal metric). Since $\gamma$ is analytic in these coordinates by Lemma \ref{YMlemma}, the function $d^2_{eucl}(q, \gamma)$  is analytic. We want to prove the quotient $\frac{d^2(q, \gamma)}{d^2_{eucl}(q, \gamma)}$ smoothly extends over $\gamma$.

We want to look at the Taylor expansion of $d^2(q, \gamma)$ at a point on $\gamma$. First change the coordinates by a diffeomorphism $\psi(u, v) = (r, s)$ by going to the normal coordinates for $\gamma$ w.r.t. the Euclidean metric (note this give an analytic chart). Then we apply Lemma \ref{geomlemma} to get that
\[d^2\big((r, s), \psi \circ \gamma\big) = c s^2 + O(|r^2 + s^2|^{\frac{3}{2}})\]
where $c > 0$ is a positive constant and $s$ is the normal variable. Therefore the quotient $D(q) := \frac{d^2(q, \gamma)}{d^2_{eucl}(q, \gamma)}$ has a smooth extension, since $d_{eucl}\big((r, s), \psi \circ \gamma\big) = s$ in these coordinates.

Also, in the $(r, s)$ coordinates, equation \eqref{F'G'} gives that $H_{1}'(r, s) D(r, s)$ is analytic and so we have $H_{1}'(u, v)D(u, v)$ also analytic, since the diffeomorphism $\psi$ is analytic, too. Finally, by going back to equation \eqref{Hextension}, we have that
\[H'(q) = F'(q)G'^{-1}(q) = \frac{F'(q) \adj G'(q)}{y^k(q) g'_1(q)} = \frac{H_1'(q) \Big(\frac{d^2(q, \gamma)}{d^2_{eucl}(q, \gamma)}\Big)^\frac{k}{2} \big(d_{eucl}^2(q, \gamma)\big)^\frac{k}{2}}{g_1'(q) \Big(\frac{d^2(q, \gamma)}{d^2_{eucl}(q, \gamma)}\Big)^\frac{k}{2} \big(d_{eucl}^2(q, \gamma)\big)^\frac{k}{2}} \]
Here $g_1' = \frac{g_1}{\det Y}$, we used \eqref{F'G'} and the $d_{eucl}$ parts cancel to give an analytic function $H'$ in the $(u, v)$ coordinates; we also applied the procedure as for \eqref{F'G'} to see that $g_1'(q) \Big(\frac{d^2(q, \gamma)}{d^2_{eucl}(q, \gamma)}\Big)^\frac{k}{2}$ is analytic. This finishes the procedure of drilling the holes.

Finally, we are left to observe that the remainder of the proof remains more or less the same as in \cite{cek1} (see also Remark 5.2. from \cite{cek1}).
\end{proof}

\begin{rem}\rm
In 2D, there are more powerful techniques to recover the connection (also true in the metric case) from the Dirichlet-to-Neumann map -- see e.g. \cite{AGTU}. However it is useful to have another viewpoint on this problem, extending the technique \cite{cek1} to this case; note also that this technique works for partial data, whereas the results of \cite{AGTU} are stated for full data.
\end{rem}



\appendix
\section{Some elementary lemmas}\label{secapp}

First, we prove an algebraic fact about harmonic polynomials.

\begin{lemma}
Assume we have four non-zero, real harmonic, homogeneous polynomials $p_{ij}$ in $\mathbb{R}^2$ for $i, j = 1, 2$ with $p_{11}p_{22} = p_{12} p_{21}$. Then one of the following two holds, up to constants and permutations:
\begin{itemize}
\item We are in the trivial case, $p_{11} = p_{12}$ and $p_{22} = p_{21}$.
\item We have $p_{22} = 1$, $p_{12} = A \re (z^k) + B \im (z^k)$ and $p_{12} = C \re (z^k) + D \im (z^k)$ for $A, B, C, D \in \mathbb{R}$ with $AC + BD = 0$ and $k \in \mathbb{N}$. Of course, then $p_{11} = p_{12} p_{21}$.
\end{itemize}
Conversely, in any of the two cases we get a quadruple of harmonic polynomials with $p_{11}p_{22} = p_{12} p_{21}$.
\end{lemma}
\begin{proof}
Let $n_{ij}$ denote the order of $p_{ij}$ for $i, j = 1, 2$. Then by recalling that the space of real harmonic, homogeneous polynomials degree $k$ is two dimensional for any $k \in \mathbb{N}$, spanned by $\re(z^k)$ and $\im(z^k)$, we observe we may write in polar coordinates $(r, \varphi)$

\[p_{ij}(z) = \re(C_{ij} z^{n_{ij}}) = r_{ij}r^{n_{ij}} \re(e^{i(\varphi_{ij} + n_{ij}\varphi)}) = r_{ij} r^{n_{ij}} \cos(\varphi_{ij} + n_{ij} \varphi)\]

Here $C_{ij} = r_{ij} e^{i \varphi_{ij}}$ are complex constants. By using the condition, we get:
\[E \cos(\varphi_{11} + n_{11} \varphi) \cos(\varphi_{22} + n_{22} \varphi) = \cos(\varphi_{12} + n_{12} \varphi) \cos(\varphi_{21} + n_{21} \varphi)\]
where $E = \frac{r_{11}r_{22}}{r_{12} r_{21}}$. Let us denote $N = n_{11} + n_{22} = n_{12} + n_{21}$. By using trigonometric formulas we may write this as:
\begin{multline}\label{eqncos}
E\big(\cos(\varphi_{11} + \varphi_{22} + \varphi N) + \cos(\varphi_{11} - \varphi_{22} + \varphi (n_{11} - n_{22}))\big)\\
= \cos(\varphi_{12} + \varphi_{12} + \varphi N) + \cos(\varphi_{12} - \varphi_{21} + \varphi (n_{12} - n_{21}))
\end{multline}
Now we use the orthogonality relations for eigenfunctions of the Laplacian on $S^1$:
\[\int_0^{2\pi} \cos(A + k\varphi) \cos(B + l\varphi) d\varphi = \begin{cases}
               0, \text{ if } k \neq \pm l\\
               \pi \cos(A - B), \text{ if } k = l \neq 0\\
               \pi \cos(A + B), \text{ if } k = -l \neq 0\\
               \pi \big(\cos(A + B) + \cos(A - B)\big), \text{ if } k = l = 0\\
            \end{cases}\]
Note that if two of $n_{ij}$ are zero, we are in the first case. We now assume $n_{ij} > 0$. By taking inner products with $\cos(\varphi_{11} + \varphi_{22} + \varphi N)$ and $\cos(\varphi_{12} + \varphi_{21} + \varphi N)$, respectively, we get
\[E \pi = \cos(\varphi_{11} + \varphi_{22} - \varphi_{12} - \varphi_{21}) \text{\,\, and \,\,} E \pi \cos(\varphi_{11} + \varphi_{22} - \varphi_{12} - \varphi_{21}) = \pi\]

and so we get $E = 1$ and $\varphi_{11} + \varphi_{22} - \varphi_{12} - \varphi_{21} = 2k\pi$ for some $k \in \mathbb{Z}$. By taking inner product with $\cos(\varphi_{11} - \varphi_{22} + \varphi (n_{11} - n_{22}))$, we see we must have $|n_{11} - n_{22}| = |n_{12} - n_{21}|$. W.l.o.g. assume $n_{11} - n_{22} = n_{12} - n_{21}$, so $n_{11} = n_{12}$ and $n_{22} = n_{21}$. Then
\[\pi = \pi \cos\big((\varphi_{12} - \varphi_{21}) - (\varphi_{11} - \varphi_{22})\big)\]

which implies $\varphi_{12} - \varphi_{11} = k_1 \pi$ and $\varphi_{22} - \varphi_{21} = k_2 \pi$, where $k_1, k_2 \in \mathbb{Z}$ are of the same parity. This then goes under the first category of solutions.

Let us now assume $n_{22} = 0$ and so $p_{22} = 1$. We are then safe to assume $n_{11}, n_{12}, n_{21} > 0$; otherwise we are in the first case trivially. In \eqref{eqncos}, take inner products with $\cos(\varphi_{12} - \varphi_{21} + \varphi(n_{12} - n_{21}))$ to get that $n_{12} = n_{21}$ (otherwise we get a contradiction with $0 \neq 0$) and so
\[\cos\big(2(\varphi_{12} - \varphi_{21})\big) = -1\]

which forces $\varphi_{12} - \varphi_{21} = \pm \frac{\pi}{2}$ (the argument range is $[0, 2\pi)$) and so the last term in \eqref{eqncos} vanishes. So after a trigonometric transformation, we get
\[2D \cos(\varphi_{11} + \varphi n_{11}) \cos(\varphi_{22}) = \cos(\varphi_{12} + \varphi_{21} + \varphi n_{11})\]
Taking further inner product, we quickly see we must have
\[\varphi_{22} - (\varphi_{12} + \varphi_{21}) = l \pi\]
for $l \in \mathbb{Z}$. Therefore we get a system of conditions:
\[\begin{cases}
               2D \cos(\varphi_{22}) &= (-1)^l\\
               \varphi_{22} &= \varphi_{12} + \varphi_{21} + l\pi\\
               n_{12} &= n_{21} = \frac{n_{22}}{2} = k\\
               \varphi_{12} - \varphi_{21} &= \pm \frac{\pi}{2}
\end{cases}\]

It is easy to check that the condition $\varphi_{12} - \varphi_{21} = \pm \frac{\pi}{2}$ gives exactly the condition on $A = r_{12} \cos(\varphi_{12})$, $B = r_{12} \sin(\varphi_{12})$, $C = r_{21}\cos(\varphi_{21})$ and $D = r_{21} \sin(\varphi_{21})$ in the second item above. Conversely, it is easy to see that the conditions in the second item, are sufficient to have a product of two harmonic homogeneous polynomials of same degree, again harmonic.
\end{proof}

\begin{rem}\rm
Note that in the case of complex harmonic polynomials, we cannot expect to have the analogous result: consider e.g. $p_{ij} = z^{n_{ij}}$ with $n_{11} + n_{22} = n_{12} + n_{21}$ -- the complex variable $z$ makes things more complicated. For other classes of ``non-trivial" examples, we have: $p_{ij} = \bar{z}^{n_{ij}}$ with the same condition on $n_{ij}$ or $p_{22} = 1$, $p_{11} = z^3 - 1$, $p_{12} = z^2 + z + 1$ and $p_{21} = z - 1$ (or slightly more generally for any polynomials in one variable with $p_{11} p_{22} = p_{12} p_{21}$, substituting $z$ would yield an example). Are these families the only possibilities?
\end{rem}

\begin{rem}\rm
We make here a few remarks about possible generalisations of the previous Lemma to higher dimensions, i.e. $n > 2$. We will focus on $n = 3$ case for simplicity, where the space of harmonic homogeneous polynomials of degree $d$ is of dimension $2d + 1$, so things complicate. 

Significant here seem to be the basis $Y^m_l$ of spherical harmonics of degree $l$, for $-l \leq m \leq l$ (these are given as restrictions of harmonic polynomials to the sphere $S^2$ and are eigenfunctions of $\Delta_{S^2}$). Similarly as in Lemma \ref{harmpoly}, after we quotient out the radial part, we are left with linear combinations of spherical harmonics. The role of products of trigonometric functions for $S^1$ is taken here by the product formula:\footnote{This follows from the representation theory of $SO(3)$ on $S^2$ -- more precisely, the proof considers the irreducible representations $H_k$ of $SO(3)$ on spherical harmonics of degree $k$ and the formula $H_k \otimes H_l \cong \bigoplus_{r = |k - l|}^{k+l}H_r$, together with a formula for $Y_l^m$ in terms of rotation matrices. See \cite{ST}, p.216 and equation 3.6.52 in the same book for a derivation.}
\begin{align}
Y_l^m Y_{l'}^{m'} = \sum_{l''} c(l'', l', m', l, m)Y_{l''}^{m''}
\end{align} 
where $m'' = m + m'$ and $|l - l'| \leq l \leq l + l''$. This might be relevant for a derivation of a version of the previous Lemma (in the real case) for $n = 3$, as we could speculate to have $|n_{11} - n_{22}| = |n_{12} - n_{21}|$, so that up to permutations ${n_{11}, n_{12}} = {n_{21}, n_{22}}$. We leave this as an open question.
\end{rem}

Now we digress from harmonic polynomials and consider a simple geometric lemma. 

\begin{lemma}\label{geomlemma}
Let $(\Omega, g) \subset \mathbb{R}^2$ containing $0$ with $g$ smooth. Fix a smooth embedded curve $0 \in \gamma \subset \Omega$. Then the Riemannian distance function $f(q) := d^2(q, \gamma)$ from a point $q \in \Omega$ to $\gamma$, has the following Taylor expansion at $0$, for $q = (x, y)$:
\begin{align}\label{geomlemmaeq}
f(x, y) = \begin{pmatrix}
x & y
\end{pmatrix} P^Tg(0)P\begin{pmatrix}
x\\
y
\end{pmatrix} + O(|x|^3)
\end{align}
where $P$ is the projection to $\star \dot{\gamma}(0)$ along $\dot{\gamma}(0)$, where $\dot{\gamma}(0)$ is the unit tangent vector to $\gamma$ at $0$.
\end{lemma}
\begin{proof}
By taking the normal coordinates for $\gamma$ with respect to $g$ and denote the new coordinate system by $(u, v) = \varphi(x, y)$ (such that $\varphi(0, 0) = (0, 0)$), then in new coordinates we know $f(u, v) = v^2$, which is coherent with \eqref{geomlemmaeq} for the metric $g' = \varphi_*g$. All we have to do now is to check this transform in the correct way back to $(x, y)$ coordinates. 

By letting $\varphi^{-1} = \psi$ and by observing that $d_{g}\big(\psi(u, v), \gamma\big) = d_{\psi^* g}\big((u, v), \varphi \circ \gamma\big) = v^2$ and differentiating with respect to $(u, v)$, we get at $(0, 0)$:
\[D^2f(\partial_u \psi, \partial_u \psi) = 0, \quad D^2f(\partial_u \psi, \partial_v \psi) = 0 \quad \text{and} \quad D^2f(\partial_v \psi, \partial_v \psi) = 2\]
This line gives the Hessian of $u$ in the basis $\psi_u, \psi_v$. From this and the formula $g = \psi^* g' = d\psi^T g' d\psi$, we may compute the Hessian in the standard basis and obtain the final result.
\end{proof}

\end{document}